\documentclass[12pt]{extarticle}
\usepackage{amsmath, amsthm,graphicx,amsfonts,amssymb,mathrsfs,showidx}
\usepackage{hyperref}
\usepackage{multicol}
\usepackage{tikz-cd} 

\usepackage{tikz}
\usetikzlibrary{tqft}

\hypersetup{
    colorlinks=true,
    linkcolor=blue,
    filecolor=blue,      
    urlcolor=cyan,
    citecolor=blue,
}

\usepackage{relsize}
\usepackage{color}
\usepackage[english]{babel} 

\definecolor{rojo}{rgb}{1,0,0}
\usepackage[all]{xy}
\makeindex
\usepackage{hyperref}
\usepackage{color}
\definecolor{abelian}{cmyk}{0.50,0,1,.4}
\definecolor{noabelian}{cmyk}{0.94,0.54,0,0}
\definecolor{rojo}{cmyk}{0,1,1,0}
\definecolor{verde}{cmyk}{0.91,0,0.88,0.12}
\xyoption{arc}

\makeatletter
\newcommand{\customlabel}[2]{%
   \protected@write \@auxout {}{\string \newlabel {#1}{{#2}{\thepage}{#2}{#1}{}} }%
   \hypertarget{#1}{#2}
}
\makeatother

\newtheorem{thm}{Theorem}

\newtheorem{lem}[thm]{Lemma}
\newtheorem{prop}[thm]{Proposition}
\newtheorem{example}{Example}[section]
\theoremstyle{definition}
\newtheorem{defn}[thm]{\textbf{Definition}}

\theoremstyle{definition}

\theoremstyle{remark}

\theoremstyle{Notation}

\newcommand{\op}{\operatorname}
\newcommand {\CC} {\mathbb{C}}
\newcommand{\ben}{\begin{equation}}
\newcommand{\een}{\end{equation}}
\newcommand{\bena}{\begin{equation*}}
\newcommand{\eena}{\end{equation*}}

\newcommand{\ma}{\mathcal}

\newcommand{\ZZ}{\mathbb{Z}}
\newcommand{\Tot}{\longmapsto}

\def\RR{\mathbb{R}}
\def\ZZ{\mathbb{Z}}
\def\FF{\mathbb{F}}

\title{Extending free actions of finite groups on unoriented surfaces}

\author{
Omar A. Cruz\thanks{Centro de Investigaci\'on en Matem\'aticas, M\'exico. 
{\em e-mail: }{\tt omar.cruz@cimat.mx}},
Gustavo Ortega \thanks{Universidad Nacional Aut\'onoma de M\'exico, M\'exico. 
{\em e-mail: }{\tt gustavo.ortega@ciencias.unam.mx}}
and 
Carlos Segovia\thanks{Instituto de Matem\'aticas UNAM-Oaxaca,  M\'{e}xico. {\em e-mail: }
{\tt csegovia@matem.unam.mx}}
}

\begin{document}
\maketitle

\begin{abstract}
We present the unoriented versions of the Schur and Bogomolov multipliers associated with a finite group $G$.
We show that the unoriented Schur multiplier is isomorphic to the second cohomology group $H^2(G;\ZZ_2)$.
We define the unoriented Bogomolov multiplier as the quotient of the unoriented Schur multiplier by the subgroup generated by classes over the disjoint union of tori, Klein bottles, and projective spaces.
We prove that the unoriented Bogomolov multiplier is trivial for abelian, dihedral, symmetric, and alternating groups. Since $H^2(G;\ZZ_2)$ is trivial for any group of odd order, there are numerous examples where the classical Bogomolov multiplier is nontrivial while its unoriented counterpart is trivial. Nevertheless, we exhibit a group of order $64$ for which the unoriented Bogomolov multiplier is nontrivial.
\end{abstract}

\section*{Introduction}
\label{intro}
The Schur and Bogomolov multipliers have been of great interest because of the solution in \cite{ASSU23} of the following question: Do all free actions on oriented surfaces equivariantly bound?
There is an affirmative answer in \cite{RZ96,DS22} for abelian, dihedral, symmetric, and alternating groups. However, we can find counterexamples in \cite{Sam22} for finite groups with non-trivial Bogomolov multiplier. 
It was shown in \cite{ASSU23} that a finite free action on a closed, oriented surface extends to a (not necessarily free) action on a $3$-manifold if and only if the corresponding class in the Bogomolov multiplier $B_0(G)$ is trivial.

\begin{thm}[\cite{ASSU23}]\label{assu}
Assume $G$ is a finite group with a free action over a closed, orientable surface $S$. This action extends to a non-necessarily free action over a $3$-manifold if and only if the class in the Bogomolov multiplier $[S, G]\in B_0(G)$ is trivial.
\end{thm}

We raise the question in the unoriented context: Are all free actions on unoriented surfaces equivariantly bound?
The obstruction to extending by free action is given by the two-dimensional unoriented bordism group $\Omega_2(BG)\cong H_2(BG;\ZZ_2)\oplus \ZZ_2$, where the second term is represented by the trivial $G$-bundle over the projective space. 
For $G$ of odd order, we have $H_2(BG;\ZZ_2)$ is trivial, and the trivial $G$-bundle over the projective space is not extendable 
even if we take non-necessarily free actions. Nevertheless, for $G$ of even order, this trivial bundle could be extended. For instance, in the case $ G=\mathbb {Z} _ 2$, we can take the action by reflection over the middle of an arc with boundary two points, and then multiply by the projective space $\mathbb{R}P^2$.

The Schur multiplier $\ma{M}(G)$ has isomorphic interpretations in terms of the free equivariant bordism $\Omega_2^{SO}(G)$, the integral homology $H_2(G,\ZZ)$, and the cohomology of groups $H^2(G,\CC^*)$; see the book of Karpilovsky for a beautiful account \cite{Kar87}.
Miller gives an important interpretation of the Schur multiplier \cite{Mi52} using the universal commutator relations: denote by $\langle G,G\rangle$ the free group on pairs $\langle x, y\rangle$ with $x,y\in G$ and consider the kernel of the canonical map to the commutator group $\langle G,G\rangle\longrightarrow [G,G]$. There is the following isomorphism for the Schur multiplier:

\begin{thm}[\cite{Mi52}]\label{miller1}
There is an isomorphism of the Schur multiplier given by
$$\ma{M}(G)\cong \frac{\op{ker}\left(\langle G,G\rangle\longrightarrow [G,G]\right)}{N}\,,$$
where $N$ is the normal subgroup generated by the following four relations: 
 \begin{eqnarray*}
		&\left<x,x\right>&\sim 1\,,\\
	&\left<x,y\right>&\sim \left<y,x\right>^{-1}\,,\\
	&\left<xy,z\right>&\sim \left<y,z\right>^x\left<x,z\right>\,,\\
	&\left<y,z\right>^x&\sim \left<x,[y,z]\right>\left<y,z\right>\,.
\end{eqnarray*}
\end{thm}

The unoriented case affords another type of relations since we have to consider the set of squares $S(G)=\langle x^2:x\in G\rangle$, which are the associated monodromies for the M\"obius band. 
We denote by $\bigl(G,G\bigr)$ the free group on pairs $\langle x, y\rangle$, $(x',y')$ and $(z)$, with $x,y,x',y',z\in G$. 
The elements $\langle x,y\rangle$ are the oriented pairs sent in $S(G)$ to the commutator $[x,y]$.
The elements $(x,y)$ are called the unoriented pairs that are sent in $S(G)$ to the unoriented commutator $\{x,y\}:=xyx^{-1}y$. The elements $(z)$ are the square elements sent in $S(G)$ to $z^2$. Consequently, we obtain a canonical map
$\bigl(G,G\bigl)\longrightarrow S(G)$,
defined by $\langle x,y\rangle\Tot [x,y]$, $(x,y)\Tot xyx^{-1}y$, and $(z)\Tot z^2$. The {\bf unoriented Schur multiplier} is the quotient  
$$\ma{M}(G;\ZZ_2):= \frac{\op{ker} \left(\bigl(G,G\bigl)\longrightarrow S(G)\right)}{N'}\,,$$
where $N'$ is the normal subgroup generated by the following relations: \\
 \begin{eqnarray*}
  		&(x^i)(x^j)&\sim (x^{i+j})\hspace{1cm} i,j\in\ZZ\,,\\
		&(x,xy)&\sim (x)(y),\\
	&\langle x,y\rangle&\sim (x)(x^{-1}y)(y^{-1}),\\
	&\left<xy,z\right>&\sim (y,z)^x(x,z^{-1})\,,\\
	&\left<y,z\right>^x&\sim \left<x,[y,z]\right>\left<y,z\right>\,,\\
        &(y,z)^x&\sim \left<x,\{y,z\}\right>(y,z)\,,\\
        &(y^{x},z^{x})^{-1}&\sim (x,\{y,z\}^{-1})(y,z)\,,\\
        & (z) \langle x,y \rangle &\sim \langle y,x \rangle^z (z[x,y])\,.
\end{eqnarray*}  
and two commutation relations:\\
\begin{align*}\label{conmuta1}
    (b)\langle a_0,b_0\rangle\sim\langle a_0,b_0\rangle^{b^2}(b)\,,\\ 
    \langle a_0,b_0\rangle (a)\sim (a)\langle a_0,b_0\rangle^{a^{-2}}\,.
\end{align*}
%
%
We can show that the oriented relations in Theorem \ref{miller1} are a consequence of the unoriented relations; see Theorem \ref{dador}. 
We obtain a Hopf formula as follows:
    $$\ma{M}(G;\ZZ_2)\cong \frac{R\cap{S(F)}}{[F,R]R^2}\,.$$
and we have an isomorphism $\ma{M}(G;\ZZ_2)\cong H^2(G;\ZZ_2)$.

The Bogomolov multiplier is defined as the quotient of the Schur multiplier $\mathcal{M}(G)$ by the toral classes provided by pairs of commuting elements. These toral classes are always extendable elements, as shown in \cite{DS22,Sam20}. 
Unlike the oriented case, we now have more extendable elements; see Proposition \ref{prC1}. More precisely, in addition to the toral classes, we have free actions whose quotient is a Klein surface, which is extendable for the same reasons as the torus, as well as principal $G$-bundles over the projective space $\RR P^2$ induced by elements $x\in G$ with $x^2=1$ and $x\neq 1$.
By dropping the subgroup generated by the trivial $G$-bundle over $\RR P^2$, we define {\bf the unoriented Bogomolov multiplier}, which we denote by $B_{0}(G;\ZZ_2)$, as the quotient of the unoriented Schur multiplier $\ma{M}(G;\ZZ_2)$ by the subgroup generated by free actions where the quotient can be a torus, a Klein bottle, or a projective space $\RR P^2$.

We show that the unoriented Bogomolov multiplier is trivial for finite abelian, dihedral, symmetric, and alternating groups. 
The most important contribution of the present paper is an example of a group where the unoriented Bogomolov multiplier is not trivial: this is a group of order $64$ which is the first group with non-trivial Bogomolov multiplier; see Example \ref{64}. 
In the forthcoming work \cite{Seg}, we will prove Theorem \ref{no-proof}, which therefore implies that such a group acts in a freely way that does not equivariantly bound, even if we allow for non-free and unoriented preserving actions.

\section*{Acknowledgment} 
We would like to thank Eric Samperton and Bernardo Uribe for relevant conversations regarding this work. 
We also thank the financial support of ``Convocatoria PAEP Posgrado de Matem\'aticas de la UNAM 2023." The third author is supported by Investigadores por M\'exico SECIHTI. We thank the anonymous reviewer for the careful reading and valuable suggestions that helped improve the manuscript.

\section{$G$-cobordism}
\label{Gcob}
A $G$-cobordism is a cobordism class of principal $G$-bundles over surfaces. 
Formally, a cobordism between two $d$-dimensional manifolds $\Sigma$ and $\Sigma'$ is represented by a $(d+1)$-dimensional manifold $M$ whose boundary satisfies $\partial M \cong \Sigma \sqcup \Sigma'$ (or $\partial M \cong \Sigma \sqcup -\Sigma'$ in the oriented case). Two cobordisms $(\Sigma, M, \Sigma')$ and $(\Sigma, M', \Sigma')$ are considered equivalent if there exists a diffeomorphism $\Phi \colon M \longrightarrow M'$ that makes the following diagram commute:
$$\xymatrix{&M\ar[dd]^\Phi&\\\Sigma \ar[ru]\ar[rd]&&\Sigma'\ar[lu]\ar[ld]\\ &M'&\,.}$$
Extending this notion to the $G$-equivariant setting is straightforward. This is done in \cite{Seg23,DS22}, where a cobordism of principal $G$-bundles $[P, Q, P']$ projects $G$-equivariantly onto a cobordism $[\Sigma, M, \Sigma']$, with $P \to \Sigma$, $Q \to M$, and $P' \to \Sigma'$ being principal $G$-bundles that agree on the restrictions and are compatible with the $G$-actions.

In dimension two, for oriented surfaces, every $G$-cobordism can be constructed by gluing principal $G$-bundles over three elementary surfaces: the cylinder, the pair of pants, and the disc. In the unoriented case, one must also include $G$-bundles over the M\"obius strip. We will describe each of these building blocks in detail, but first, we recall the classification of principal $G$-bundles over the circle.

Let $x \in G$. Consider the product space $[0,1] \times G$ with ends identified via multiplication by $x$, i.e., $(0, y) \sim (1, xy)$ for all $y \in G$. The resulting space is denoted $P_x$. Every principal $G$-bundle over the circle is isomorphic to some $P_x$, and $P_x \cong P_{x'}$ if and only if $x'$ is conjugate to $x$. The projection $P_x \to S^1$ is given by the first coordinate, and the $G$-action is right multiplication on the second coordinate. Throughout this paper, we refer to $x$ as the \textbf{monodromy} of $P_x$. Note that $P_1$ corresponds to the trivial bundle.

The elementary $G$-cobordisms for unoriented surfaces are as follows:

\begin{itemize}
\item \textbf{Cylinder.} Every principal $G$-bundle over the cylinder, with incoming monodromy $x$ and outgoing monodromy $x'$, corresponds to an element $y \in G$ such that $x' = y x y^{-1}$.

\item \textbf{Pair of pants.} For $x,y\in G$, we consider the principal $G$-bundle over the pair of pants, which is a $G$-deformation retract\footnote{By a $G$-deformation retract we mean a homotopy realized through principal $G$-bundles.} of a principal G-bundle over the wedge $S^1\vee S^1$. The incoming boundaries carry monodromies $x$ and $y$, while the outgoing boundary has monodromy $xy$.

\item \textbf{Disc.} There is a unique $G$-cobordism over the disc, given by the trivial bundle.

\item \textbf{M\"obius strip.} For $x \in G$, the middle circle of the M\"obius strip carries monodromy $x$, and the boundary circle carries monodromy $x^2$.

\end{itemize}
Figure \ref{figr1} illustrates these $G$-cobordisms. Circles are labeled with their corresponding monodromy, and for cylinders, the interior indicates the conjugating element.
\begin{figure}
\centering
\includegraphics[scale=0.9]{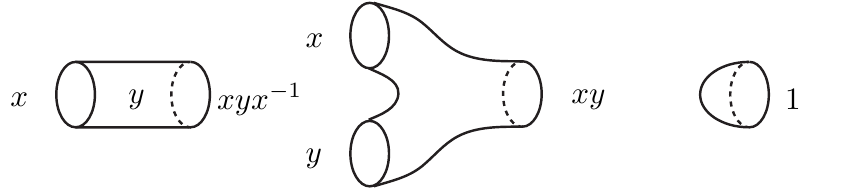}
\caption{Elementary $G$-cobordisms over the cylinder, pair of pants, and disc.}
\label{figr1}
\end{figure}
Figure \ref{figr2} depicts the $G$-cobordism over the M\"obius strip and its schematic representation in diagrams.

\begin{figure}
    \centering
    \includegraphics[scale=0.5]{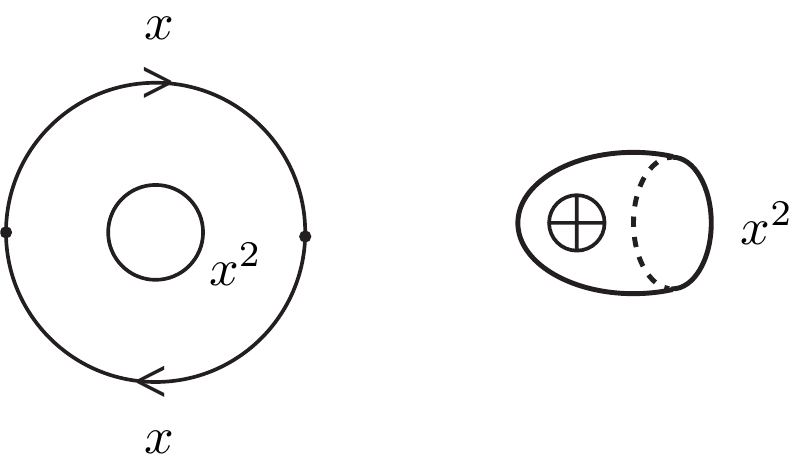}
    \caption{A $G$-cobordism over the M\"obius strip.}
    \label{figr2}
\end{figure}

\begin{example}[The projective space $\RR P^2$]
Recall that $\mathbb{R} P^2$ with an open disc removed is a M\"obius strip. Let $x \in G$ be the monodromy along the middle circle of the M\"obius strip. Capping its boundary with a disc (which carries only the trivial $G$-bundle) forces $x^2 = 1$. Figure \ref{figr3} illustrates how a $G$-cobordism over $\mathbb{R} P^2$ is obtained from a M\"obius strip and a disc.
\begin{figure}
    \centering
    \includegraphics[scale=0.15]{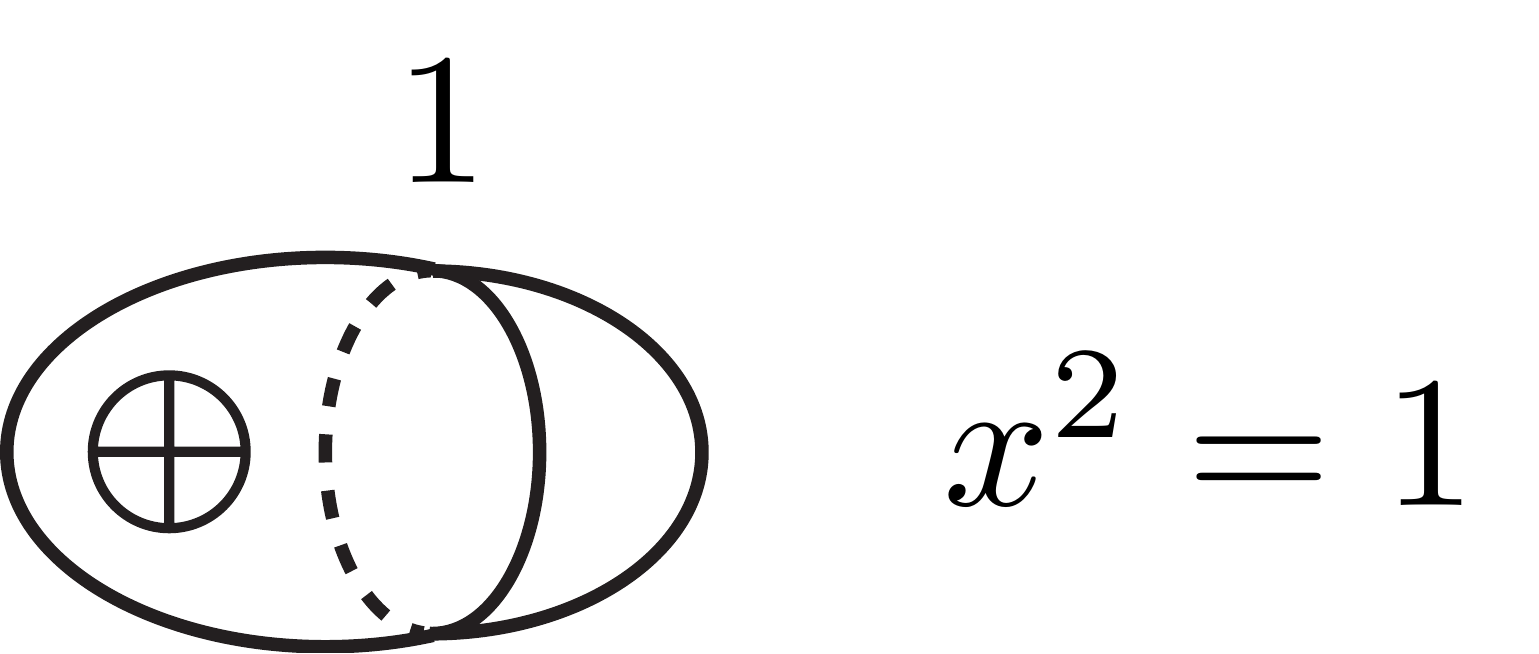}
    \caption{A $G$-cobordism over $\RR P^2$.}
    \label{figr3}
\end{figure}
\end{example}

\begin{example}[A handle of genus $n$] 
As in \cite{DS22}, a $G$-cobordism over a handlebody of genus $n$ with one boundary circle depends on elements $x_i, y_i \in G$ for $1 \leq i \leq n$, with boundary monodromy $\prod_{i=1}^n [x_i, y_i]$. Figure \ref{figr4} shows two equivalent representations for genus one.
\begin{figure}
    \centering
   \includegraphics[scale=0.9]{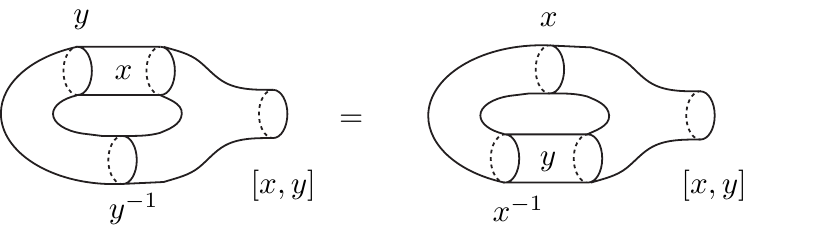}
    \caption{Two equivalent representations of a $G$-cobordism over a handle of genus one.}
    \label{figr4}
\end{figure}
\end{example}

\begin{example}[The Klein bottle with a boundary circle]
A $G$-cobordism over a Klein bottle depends on two elements $x, y \in G$, with boundary monodromy $\{x, y\} := x y x^{-1} y$. Two equivalent representations are shown in Figure \ref{figr5}.
\begin{figure}
    \centering
    \includegraphics[scale=0.9]{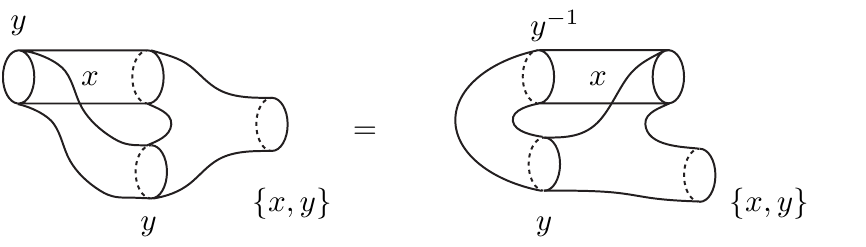}
    \caption{Two equivalent representations of a $G$-cobordism over a Klein bottle.}
    \label{figr5}
\end{figure}

\end{example}

Certain relations hold among $G$-cobordisms over the M\"obius strip and the Klein bottle. For instance, Figure \ref{figr6} shows that a Klein bottle $G$-cobordism decomposes as the union of two M\"obius strip $G$-cobordisms.

\begin{figure}
    \centering
    \includegraphics[scale=0.9]{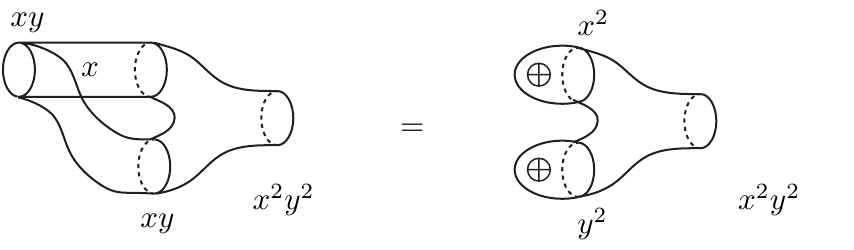}
    \caption{The Klein bottle as the union of two M\"obius strips.}
    \label{figr6}
\end{figure}

We now use the notion of $G$-cobordism to characterize when a free $G$-action on a surface bounds equivariantly.

\begin{defn}
A two-dimensional $G$-cobordism $[Q \to M]$ with $\partial Q = \partial M = \emptyset$ is called \textbf{extendable} if there exists a compact $3$-manifold $W$ with $\partial W \cong Q$ and an action of $G$ on $W$ that extends the action on $Q$.
\end{defn}

Examples of extendable $G$-cobordisms include trivial bundles over orientable surfaces, which extend to solid handlebodies. A non-example is the trivial bundle over $\mathbb{R} P^2$ when $G$ is of odd order. The following proposition generalizes extension constructions to certain unoriented surfaces.

\begin{prop}\label{prC1}
Every $G$-cobordism whose quotient is a torus or a Klein bottle is extendable. Moreover, any $G$-cobordism over $\mathbb{R} P^2$ associated with an element $x \in G$ satisfying $x^2 = 1$ and $x \neq 1$ is also extendable.
\end{prop}

\begin{proof}
The torus case is treated in \cite{Sam20,DS22} via radial extension to a solid handlebody. One places the initial $G$-cobordism at radius $r = 1$ and shrinks it continuously as $r \to 0$, ending at a branched circle. The Klein bottle case is analogous, but the extension yields an unoriented $3$-manifold.

For $\mathbb{R} P^2$, observe that removing a disc yields a M\"obius strip. Restrict the given $G$-cobordism to this M\"obius strip, and let $x \in G$ be the monodromy along its middle circle, where $x^2 = 1$. For $x \neq 1$, the total space of the bundle over the M\"obius strip consists of $|G|/2$ cylinders, each doubly covering the strip. Capping each cylinder with two discs along its boundaries yields $|G|/2$ spheres, which together constitute a principal $G$-bundle over $\mathbb{R} P^2$; see Figure \ref{Mobius}.

\begin{figure}
    \centering
    \includegraphics[scale=0.7]{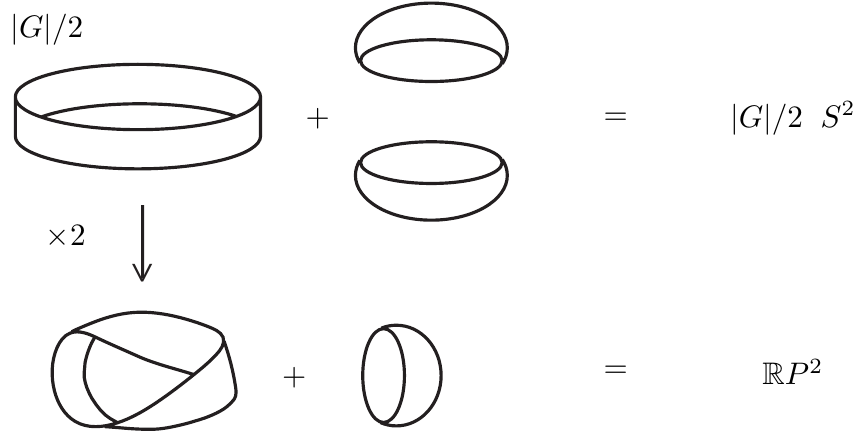}
    \caption{The bundles over $\RR P^2$ with monodromy $x\neq 1$.}
    \label{Mobius}
\end{figure}

The bundle in Figure \ref{Mobius} extends radially to a $3$-manifold formed by $|G|/2$ closed $3$-balls, whose boundaries are the spheres just described. The configuration is preserved for each radius $r > 0$, and the centers of the balls are fixed points of the $G$-action.
\end{proof}

An immediate consequence is the case of abelian groups.

\begin{thm}\label{abeliano}
If $G$ is abelian, then every $G$-cobordism over an unoriented surface is extendable. 
\end{thm}
\begin{proof}
The oriented case is established in \cite[Theorem 8]{DS22}. For non-orientable surfaces, we recall their classification: they are either the connected sum of a projective plane with $(n-1)/2$ tori (for $n$ odd), or the connected sum of a Klein bottle with $(n-2)/2$ tori (for $n$ even). In a $G$-cobordism over such a surface, the monodromy along the curve separating the tori from the projective plane or Klein bottle lies in the commutator subgroup of $G$, which is trivial since $G$ is abelian. Cutting along this curve and capping with discs separates the tori from the remaining piece. Applying Proposition \ref{prC1} completes the proof.
\end{proof}

\section{The unoriented Schur multiplier}
\label{schur}

Issai Schur \cite{Sch04} introduced a mathematical object to classify projective representations $\rho \colon G \longrightarrow \operatorname{Gl}(n,\mathbb{C})/Z$, where $Z$ denotes the center of scalar matrices, using the quotient of $2$-cocycles by the group of $2$-coboundaries. Initially termed the Multiplikator of $G$, it is now known as the Schur multiplier $\mathcal{M}(G)$. We refer to \cite{Wi81} for a historical overview and discussion. Over time, this object acquired interpretations in terms of group homology and cohomology \cite{Br82}, and it later became significant in low-dimensional topology, where it classifies free actions of finite groups on closed oriented surfaces (see \cite{Ed82,Ed83}). In summary, $\mathcal{M}(G)$ can be identified with three equivalent groups: the second homology group $H_2(G,\mathbb{Z})$, the second cohomology group $H^2(G,\mathbb{C}^*)$, and the second free oriented bordism group $\Omega_2^{SO}(G)$.

Clair Miller, a student of Spanier, described the Schur multiplier $\mathcal{M}(G)$ in terms of universal commutator relations \cite{Mi52}. Let $\langle G,G\rangle$ denote the free group on pairs $\langle x, y \rangle$ with $x, y \in G$, and let $K := \ker\bigl( \langle G,G\rangle \to [G,G] \bigr)$ be the kernel of the canonical map to the commutator subgroup. The Schur multiplier $\mathcal{M}(G)$ is isomorphic to the quotient of $K$ by the normal subgroup generated by the following four relations:

\setcounter{equation}{0}
 \begin{eqnarray}\label{four1}
		&\left<x,x\right>&\sim 1\,,\\\label{four2}
	&\left<x,y\right>&\sim \left<y,x\right>^{-1}\,,\\\label{four3}
	&\left<xy,z\right>&\sim \left<y,z\right>^x\left<x,z\right>\,,\\\label{four4}
	&\left<y,z\right>^x&\sim \left<x,[y,z]\right>\left<y,z\right>\,,
\end{eqnarray}
where $x,y,z\in G$ and $\langle y,z\rangle^x=\langle y^x,z^x\rangle = \langle xyx^{-1}, xzx^{-1}\rangle$.

Miller also derived several further relations from these, which we list in the following theorem.
\begin{thm}[\cite{Mi52}]
The following relations can be deduced from \eqref{four1}-\eqref{four4}:
 \begin{eqnarray}\label{four5}
 &\left<x,yz\right>&\sim \left<x,y\right>\left<x,z\right>^y\,,\\\label{four6}
 &\left<x,y\right>^{\left<a,b\right>}&\sim \left<x,y\right>^{[a,b]}\,,\\\label{four7}
 & \left[\left<x,y\right>,\left<a,b\right>\right]&\sim \left<[x,y],[a,b]\right>\,,\\\label{four8}
 & \left<b,b'\right>\left<a_0,b_0\right> &\sim \left<[b,b'],a_0\right>\left<a_0,[b,b']b_0\right>\left<b,b'\right>\,,\\\label{four9}
 & \left<b,b'\right>\left<a_0,b_0\right> &\sim \left<[b,b']b_0,a_0\right>\left<a_0,[b,b']\right>\left<b,b'\right>\,,\\\label{four10}
 & \left<b,b'\right>\left<a,a'\right> &\sim \left<[b,b'],[a,a']\right>\left<a,a'\right>\left<b,b'\right>\,,\\\label{four11}
 &\left<x^n,x^s\right> &\sim 1\hspace{1cm}n=0,\pm1,\cdots;s=0,\pm1,\cdots\,,
 \end{eqnarray}for $x,y,z,a,b,a',b',a_0,b_0\in G$.
\end{thm}
In the unoriented setting, Miller's definition involves the subgroup $S(G) = \langle x^2 : x \in G \rangle$, as in \cite{HM97}. Let $\bigl(G,G\bigr)$ denote the free group on symbols $\langle x, y \rangle$, $(x',y')$, and $(z)$ with $x,y,x',y',z \in G$. The pairs $(x,y)$ are called unoriented pairs and are mapped under $S(G)$ to the \textbf{unoriented commutator} $\{x,y\} := x y x^{-1} y$, while the elements $(z)$ are called square elements and are sent to $z^2$. This yields a canonical homomorphism
$$(G,G)\rightarrow S(G)\,,$$
defined by $\langle x,y \rangle \mapsto [x,y]$, $(x,y) \mapsto \{x,y\}$, and $(z) \mapsto z^2$.
\begin{defn}
\label{Schurito}
Let $K' := \ker\bigl( \bigl(G,G\bigr) \to S(G) \bigr)$ be the kernel of the canonical map. The unoriented Schur multiplier, denoted $\mathcal{M}(G; \mathbb{Z}_2)$, is defined as the quotient of $K'$ by the normal subgroup generated by the following relations:
  \begin{eqnarray}\label{four12}
  		&(x^i)(x^j)&\sim (x^{i+j})\hspace{1cm} i,j\in\ZZ\,,\\\label{four13}
		&(x,xy)&\sim (x)(y),\\\label{four14}
	&\langle x,y\rangle&\sim (x)(x^{-1}y)(y^{-1}),\\\label{four15}
	&\left<xy,z\right>&\sim (y,z)^x(x,z^{-1})\,,\\\label{four16}
	&\left<y,z\right>^x&\sim \left<x,[y,z]\right>\left<y,z\right>\,,\\\label{four17}
        &(y,z)^x&\sim \left<x,\{y,z\}\right>(y,z)\,,\\\label{four18o}
        &(y^{x},z^{x})^{-1}&\sim (x,\{y,z\}^{-1})(y,z)\,,\\\label{fournew}
        &(z) \langle x,y \rangle &\sim \langle y,x \rangle^z (z[x,y])\,.\label{fournew2}
\end{eqnarray}  
\end{defn}
%
%
The relation \eqref{fournew2} appears in \cite{KT17} in the context of equivariant Klein TQFT.
In addition, we have two commutation relations:\\
\begin{eqnarray}\label{conmmuta1}
    (b)\langle a_0,b_0\rangle\sim\langle a_0,b_0\rangle^{b^2}(b)\,,\\ \label{conmmuta2}
    \langle a_0,b_0\rangle (a)\sim (a)\langle a_0,b_0\rangle^{a^{-2}}\,.
\end{eqnarray}
These commutation relations correspond to two equivalent ways of describing the $G$-cobordism over the pair of pants with inputs $x$ and $y$; see Figure~\ref{fig-final} for an illustration.
\begin{figure}
\begin{center}
    \hspace{2cm}\includegraphics[width=0.7\linewidth]{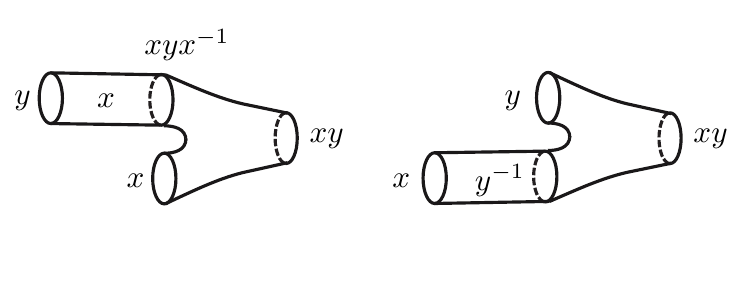}
 \end{center}   
    \caption{Two equivalent ways of describing the $G$-cobordism over the pair of pants with inputs $x$ and $y$ and output $xy$.}
    \label{fig-final}
\end{figure}

Similarly, one can derive a series of further identities from these defining relations.
\begin{thm}\label{dador}
The relations \eqref{four12}-\eqref{four18o} imply the following:
 \begin{eqnarray}\label{four18}
 (1)\sim 1,&(x^{-1})\sim(x)^{-1}\,,&(x,1)\sim 1\\\label{four19}
  \langle x,x\rangle\sim 1,&(x,x)\sim (x)\sim (1,x)\,,&\langle x,y\rangle \sim\langle y,x\rangle^{-1}\\\label{four20}
(x,y)^{-1}\sim (x^{-1},y^{-1})^x,&(xy,z)\sim (y,z)^x\langle x,z^{-1}\rangle,&\\\label{four21}
(xy,z)\sim\langle y,z\rangle^x(x,z)&\langle xy,z\rangle\sim\langle y, z\rangle^x\langle x,z\rangle,&\\\label{four22}
(z_1z_2)\sim\langle z_1^{-1},z_1z_2\rangle^{z_1}(z_1)(z_2),&(z_1)(z_2)\sim\langle z_1,z_1z_2\rangle(z_1z_2),&\\\label{four23}(z_1z_2)\sim\langle z_1z_2,z_1\rangle (z_1)(z_2),&(z_1)(z_2)\sim\langle z_1z_2,z_1^{-1}\rangle^{z_1}(z_1z_2),&\\\label{four24}
(y,z)\sim\langle y,z\rangle (z),& (y,z)\sim(z)^y\langle y,z^{-1}\rangle,&\\\label{four25}
(xy,x)\sim (y,x)^x,&(a)(b)\sim(b^{-1})^a(ab^2),&\\\label{four26}
(z)(x,y)\sim (x^{-1},y^{-1})^{zx}(z\{x,y\}),&\langle x,y\rangle^2\sim([x,y])\,.&\label{fournew3}
 \end{eqnarray}for all $x,y,z,z_1,z_2,a,b\in G$.
\end{thm}
\begin{proof}
Identities \eqref{four18} and \eqref{four19} follow directly from \eqref{four12}, \eqref{four13}, and \eqref{four14}. Those in \eqref{four20} and \eqref{four21} are consequences of \eqref{four15}. For the left-hand side of \eqref{four21}, substitute $x \mapsto x^{-1}$ in \eqref{four15} to obtain
 $$(y,z)^{x^{-1}}\sim \langle x^{-1}y,z\rangle (x^{-1},z^{-1})^{-1}\stackrel{\eqref{four20}}{\sim} \langle x^{-1}y,z\rangle (x,z)^{x^{-1}}$$
 Applying conjugation by $x$ yields 
 $$(y,z)\sim\langle x^{-1}y,z\rangle^x(x,z)$$
 and a suitable change of variables gives the desired relation.  Identities \eqref{four22} and \eqref{four23} follow from \eqref{four15}. For instance, setting $x = z_1$, $y = z_1^{-1}$, and $z = z_1 z_2$ in the left-hand side of \eqref{four21} yields the left-hand side of \eqref{four22}. Identity \eqref{four24} is deduced from \eqref{four15}. The left-hand side of \eqref{four25} follows from the right-hand side of \eqref{four20}. 
 The right-hand side of \eqref{four25} follows by combining the left-hand side with \eqref{four13}:
$$(xy)(y^{-1})\sim (y)^x(y^{-1}x)^x\,.$$
The substitution of $a=xy$, $b=y^{-1}$ implies the identity. 
Finally, the left-hand side from \eqref{four26} is obtained from the right-hand side of \eqref{four25} and right-hand side from \eqref{fournew3} is obtained from relation \eqref{fournew2}.
\end{proof}

We now compute the unoriented Schur multiplier for free, cyclic, and dihedral groups.

 \begin{prop}\label{free1} 
 The unoriented Schur multiplier of the infinite cyclic group is trivial, i.e.,
     $$\ma{M}(\mathbb{Z};\ZZ_2)\cong 0\,.$$
 \end{prop}

\begin{proof}
Let $\mathbb{Z}$ be generated by $x$. By \eqref{four11}, we have $\langle x^n, x^m \rangle \sim 1$ for all $n,m \in \mathbb{Z}$. The remaining generators are of the form $(x^s, x^t)$ and $(x^t)$. Using \eqref{four12} and \eqref{four13}, we obtain $(x^s, x^t) \sim (1, x^t) \sim (x^t)$ for all $s,t \in \mathbb{Z}$. Thus, every element in $\mathcal{M}(\mathbb{Z};\mathbb{Z}_2)$ can be written as a product $(x^{i_1}) \cdots (x^{i_n})$ with $x^{2i_1 + \cdots + 2i_n} = 1$. If $x \neq 1$, this forces $i_1 + \cdots + i_n = 0$, and hence $(x^{i_1}) \cdots (x^{i_n}) \sim (x^{i_1 + \cdots + i_n}) \sim 1$. Therefore, all elements are trivial.
\end{proof}

The following result is key for establishing a Hopf formula for the unoriented Schur multiplier in Section \ref{hopf}.

\begin{thm}\label{freezone}
The unoriented Schur multiplier of a free group is trivial. 
\end{thm}

\begin{proof}
As in the oriented case, the argument reduces to groups with finitely many generators. Proposition~\ref{free1} handles the one-generator case. Now let $G = A * B$ with $|A|, |B| < \infty$, and assume $\mathcal{M}(A;\mathbb{Z}_2) = 1$ and $\mathcal{M}(B;\mathbb{Z}_2) = 1$. 

Firstly, using relations \eqref{four13} and \eqref{four14} together with \eqref{four22} and \eqref{four23}, we obtain that any element of $(G,G)$ can be written as a product of factors of the form $(a)$, $(a'b)$, and $(b)$, with $a,a' \in A$, $b \in B$. Nevertheless, using \eqref{four22} and \eqref{four23} again, we obtain that every element in $(G,G)$ is the product of elements of $(a)$, $(b)$, and $\mu$, where $\mu$ lies in the subgroup of $(G,G)$ generated by all $\langle a,b\rangle$ with $a \neq 1$ in $A$ and $b \neq 1$ in $B$.

Similarly to the arguments in \cite{Mi52} (or see \cite[Lemma 2.6.2, pp. 68-69]{Kar87}), we need the commutation identities similar to \eqref{four8}, \eqref{four9}, and \eqref{four10}, together with repeated use of the rules:
\begin{eqnarray}\label{rel1}
    \langle a,a'\rangle^{b_0}\sim\langle b_0,[a,a']\rangle\langle a,a'\rangle\\\label{rel2}
    \langle a,b\rangle^{a_0}\sim \langle a_0a,b\rangle\langle b,a_0\rangle\\\label{rel3}
    \langle a,b\rangle^{b_0}\sim \langle b_0,a\rangle\langle a,b_0b\rangle
\end{eqnarray}
and three similar rules, obtained from these by interchanging $a$ with $b$, $a_0$ with $b_0$, and $a'$ with $b'$ (see \cite[p. 69]{Kar87}). 

Recall the commutation relations \eqref{conmmuta1} and \eqref{conmmuta2} given by $(b)\langle a_0,b_0\rangle\sim\langle a_0,b_0\rangle^{b^2}(b)$ and $\langle a_0,b_0\rangle (a)\sim (a)\langle a_0,b_0\rangle^{a^{-2}}$. For the commutation between elements $(a)$ and $(b)$, we use:
\begin{align}
    (a)(b)(a^{-1})(b^{-1})&\sim\langle a,ab\rangle(ab)(a^{-1})(b^{-1})\\
    &\sim \langle a,ab\rangle\langle ab,aba^{-1}\rangle (aba^{-1})(b^{-1})\\
    &\sim \langle a,ab\rangle\langle ab,aba^{-1}\rangle\langle aba^{-1},[a,b]\rangle ([a,b])\\\label{equidos}
    &\sim \langle a,ab\rangle\langle ab,aba^{-1}\rangle\langle aba^{-1},[a,b]\rangle \langle a,b\rangle^2\\
&\sim\langle a,b\rangle^a\langle b,a^{-1}\rangle^{aba}\langle a^{-1},b^{-1}\rangle^{ab^2}\langle a,b^{-1}\rangle^{aba^{-1}}\langle a,b\rangle^2    
\end{align}
where we used the right-hand side of \eqref{fournew3} in relation \eqref{equidos}. Consequently, by applying relations \eqref{rel1}, \eqref{rel2}, \eqref{rel3}, and the other three, we obtain the last commutation:
\begin{equation}\label{rel4}
(b)(a)\sim \mu'(a)(b)
\end{equation}
where $\mu'$ belongs to the subgroup of $(G,G)$ generated by all $\langle a,b\rangle$ with $a \neq 1$ in $A$ and $b \neq 1$ in $B$.

Using the commutation relations \eqref{conmmuta1}, \eqref{conmmuta2}, and \eqref{rel4}, together with the simplification of the conjugations \eqref{rel1}, \eqref{rel2}, \eqref{rel3}, and the other three, we arrive at the conclusion that every word $\omega \in (G,G)$ can be reduced to $\omega \sim \alpha \mu \beta$, where $\alpha$ is a word composed of elements $(a)$, $\langle a',a''\rangle$, with $a,a',a'' \neq 1$; $\beta$ is a word composed of elements $(b)$, $\langle b',b''\rangle$, with $b,b',b'' \neq 1$; and $\mu$ lies in the subgroup of $(G,G)$ generated by all $\langle a,b\rangle$ with $a \neq 1$ in $A$ and $b \neq 1$ in $B$. As in Miller's proof, we obtain $[\omega] \neq 1$ whenever $\omega$ is nonempty, proving the theorem.
\end{proof}

The cyclic group case differs from the classical Schur multiplier.

\begin{prop}\label{cyclic}
The unoriented Schur multiplier of the cyclic group $\mathbb{Z}_n$ is given by
    \begin{equation}
        \ma{M}(\ZZ_n;\ZZ_2)\cong\left\{\begin{array}{cc}
            0 &  n \textrm{ odd\,,} \\
            \ZZ_2 & n \textrm{ even\,.}  
        \end{array}\right.
    \end{equation}
where for $n$ even, a generator is represented by $(x^{n/2})$.
\end{prop}
\begin{proof}
Let $\mathbb{Z}_n = \langle x \rangle$. All symbols $\langle x^i, x^j \rangle$ are trivial by \eqref{four11}. Using \eqref{four13} and \eqref{four12}, we obtain $(x^i, x^j) \sim (x^i)(x^{j-i}) \sim (x^j)$. Hence $(x^j)$ lies in $\mathcal{M}(\mathbb{Z}_n; \mathbb{Z}_2)$ if and only if $2j \equiv 0 \pmod n$, which occurs precisely when $n$ is even and $j = n/2$.
\end{proof}

The dihedral group of order $2n$ has presentation 
$$D_{2n}=\langle a,b : a^2=b^2=(ab)^n=1\rangle\,.$$
Set $c:=ab$, the rotation by $2\pi/n$,  and note that $a c^i = c^{-i} a$.
\begin{prop}\label{dihedral}
Let $n$ be an odd integer. Then $\mathcal{M}(D_{2n}; \mathbb{Z}_2)$ is multiplicatively generated by the classes $x_1, \dots, x_n$, where $x_i = (a c^i)$. If $n$ is even, $\mathcal{M}(D_{2n}; \mathbb{Z}_2)$ is generated by $x_1, \dots, x_n$ and an additional class $y = (c^{n/2})$.
\end{prop}

\begin{proof}
By \cite{DS22}, any product $\langle x_1, y_1 \rangle \cdots \langle x_k, y_k \rangle$ is equivalent to a product of terms of the form $\langle c^i, a \rangle$. Relation \eqref{four14} gives $\langle c^i, a \rangle \sim (c^i)(a c^i)(a)$. Moreover, using \eqref{four13} we obtain:
  \begin{itemize}
      \item[i)] $(c^i,c^j)\sim (c^j)$,
      \item[ii)] $(c^i,ac^j)\sim (c^i)(ac^{i+j})$,
      \item[iii)] $(ac^i,c^j)\sim (ac^i)(ac^{i+j})$,
      \item[iv)] $(ac^i,ac^j)\sim (ac^i)(c^{j-i})$.
  \end{itemize}
Thus any element of $(G,G)$ reduces to a product of symbols $(c^i)$ and $(a c^j)$ with $i,j \in {0,\dots,n-1}$. Since $(a c^i)(a c^i) \sim (a c^i a c^i) \sim 1$, each $x_i = (a c^i)$ has order at most $2$. For products of the form $(c^{i_1}) \cdots (c^{i_m})$, the condition $\sum i_j \equiv 0 \pmod n$ must hold. When $n$ is odd, the classes $x_i$ suffice. For even $n$, the extra class $y = (c^{n/2})$ is required. Note that further relations among the $x_i$ and $y$ may exist.
  
  Therefore, any sequence in $(G,G)$ reduces to a sequence given by the product of elements of the form $(c^i)$ and $(ac^j)$ for $i,j\in \{0,\cdots,n-1\}$.  
  Since every element $(ac^i)$ satisfies to be of torsion two because $(ac^i)(ac^i)\sim (ac^iac^i)\sim 1$, in the sequence, it rests to consider the product of elements $(c^{i_1}\cdots c^{i_m})$ with $\sum_j i_j=0\mod n$. If $n$ is odd, it is enough to consider the classes $x_i=(ac^{i})$. However, if $n$ is even, we must consider an additional class given by $y=(c^{n/2})$. We want to remark that there could be relations between the classes $ x_i$ and $y$.
\end{proof}
  
\section{The Hopf formula for the unoriented case}
\label{hopf}
Assume $G$ is a finite group with presentation $G=\langle F| R\rangle$.
Miller \cite{Mi52} shows that the Schur multiplier $\ma{M}(G)$ admits an isomorphism with the Hopf formula, i.e.,
\begin{equation}\label{hopffor}\ma{M}(G)\cong \frac{[F,F]\cap R}{[F,R]}\,.\end{equation}
For the unoriented Schur multiplier $\ma{M}(G;\ZZ_2)$, we obtain the following analogous Hopf formula:
\begin{equation}\label{Hopfi}\ma{M}(G;\ZZ_2)\cong \frac{S(F)\cap  R}{{[F,R]R^2}}\,.\end{equation}
We now reproduce the process of Miller \cite{Mi52}, adapted to the unoriented case.

Recall that the unoriented Schur multiplier is defined as the quotient of the kernel $K' := \ker\bigl((G,G) \to S(G)\bigr)$ by the normal subgroup $N'$ generated by relations \eqref{four12}--\eqref{four18o}.
Given a square-central extension $1\rightarrow A\rightarrow E\stackrel{\eta}{\rightarrow} G\rightarrow 1$ (see Definition \ref{squarecen}), we define a homomorphism $(G,G) \to E$ by mapping the generators $\langle x, y \rangle$, $(x', y')$, and $(z)$ to $[\overline{x}, \overline{y}]$, ${\overline{x}', \overline{y}'}$, and $\overline{z}^2$, respectively, where $\eta(\overline{x}) = x$, $\eta(\overline{y}) = y$, $\eta(\overline{x}') = x'$, $\eta(\overline{y}') = y'$, and $\eta(\overline{z}) = z$.
Because the extension is square-central, this homomorphism is well defined.
Moreover, its restriction to $K'$ has image contained in $S(E) \cap A$, and it sends $N'$ to $1$, inducing a surjective homomorphism $\ma{M}(G;\ZZ_2) \to S(E) \cap A$. Consequently, we obtain the exact sequence
\begin{equation}\label{fifon}
\ma{M}(E;\ZZ_2) \to \ma{M}(G;\ZZ_2) \to S(E) \cap A\,.
\end{equation}

Now consider the central extension
\begin{equation}\label{gigon}
1 \to \frac{R}{[F,R]R^2} \to \frac{F}{[F,R]R^2} \to G \to 1.
\end{equation}
Set $R^0 = R/[F,R]R^2$ and $F^0 = F/[F,R]R^2$. Observe that every element of $R^0$ has order two.
From \eqref{fifon}, we obtain the exact sequence
$$\ma{M}(F^0;\ZZ_2)\rightarrow \ma{M}(G;\ZZ_2)\rightarrow S(F^0)\cap R^0\,.$$
As in the oriented case \cite[p. 594]{Mi52}, the unoriented Schur multiplier of a free group is trivial (Theorem \ref{freezone}), so $\ma{M}(F^0;\ZZ_2) = 0$.
Thus $\ma{M}(G;\ZZ_2)$ is isomorphic to $S(F^0) \cap R^0$. However, we have $S(F^0) \cap R^0 = (S(F) \cap R)/[F,R]R^2$.

A remarkable result for the classical Schur multiplier $\ma{M}(G)$ is its isomorphism with the cohomology group $H^2(G,\CC^*)$ in terms of central group extensions. The unoriented Schur multiplier $\ma{M}(G;\ZZ_2)$ also admits a cohomological interpretation, where the role of the units $\CC^*$ is played by the field with two elements $\FF_2$ (the coefficients $\ZZ_2$ endowed with multiplicative structure). We now establish an isomorphism
$$\ma{M}(G;\ZZ_2)\cong H^2(G;\ZZ_2)\,.$$

\begin{defn}\label{squarecen}
Let $G$ be a group and $A$ an abelian group. A \textit{central extension} of $A$ by $G$ is a short exact sequence
$$1\rightarrow A\rightarrow E \rightarrow G\rightarrow 1\,,$$
such that the image of $A$ lies in the center of $E$. If every element of $A$ is of order two, we call it a \textit{square-central extension}.
Two extensions $E$ and $E'$ are \textit{equivalent} if there exists an isomorphism between the two exact sequences.
The set of equivalence classes of central extensions of $A$ by $G$ is denoted by $\xi(G,A)$, and we denote by $\xi^s(G,A)$ the subset of equivalence classes of square-central extensions of $A$ by $G$.
\end{defn}

We require several results to show that $\xi^s(G,\FF_2)$ is isomorphic to the Hopf formula \eqref{Hopfi}.
The first result is a general property of central extensions; see \cite[p. 94]{Br82}.

\begin{lem}\label{lemma 15}
Consider a central extension
$$1 \rightarrow M \rightarrow E \rightarrow E/M \rightarrow 1\,.$$
Then for any abelian group $J$, there is an exact sequence
\begin{equation}
\op{Hom}(E,J)\stackrel{\op{Res}}{\longrightarrow}\op{Hom}(M,J)\stackrel{\op{Tra}}{\longrightarrow}\xi(E/M,J)
\end{equation}
where $\op{Res}$ is the restriction map and $\op{Tra}$ is the transgression map. For $\varphi: M \to J$, the transgression map sends $\varphi$ to the equivalence class of the central extension
\begin{equation}1 \rightarrow J \rightarrow \frac{J \times E}{K}\rightarrow E/M \rightarrow 1\,,
\end{equation}
where $K$ is the normal subgroup of $J\times E$ generated by $\lbrace (\varphi(m),m^{-1})|m\in M\rbrace$.
\end{lem}

\begin{proof}
Note that $\varphi \in \operatorname{im}(\operatorname{Res})$ if and only if $\varphi$ extends to $E$. Moreover, $\varphi \in \ker(\operatorname{Tra})$ if and only if there exists an isomorphism $f: \frac{J \times E}{K} \to J \times E/M$ making the following diagram commute:
 \begin{equation}\label{tyty}
 \begin{tikzcd}
1 \arrow[r] & J \arrow[r, "i"] \arrow[rd, "i'"] & J\times E/M \arrow[r, "\eta"]         & E/M \arrow[r] & 1 \\
            &                                      & \frac{J \times E}{K} \arrow[ru, "\eta'"] \arrow[u, "f"] &               &  
\end{tikzcd}
\end{equation}
We show the equivalence of these two statements. First, suppose such an $f$ exists. Define $\psi: E \to J$ by $\psi(e) = p_J \circ f((1,e)K)$, where $p_J$ is projection onto $J$. For $e \in M$, we have $(1,e)K = (\varphi(e),1)K$ by definition of $K$. Using the left commutative triangle of \eqref{tyty}, we obtain $\psi(e) = p_J \circ f((\varphi(e),1)K) = \varphi(e)$. Conversely, suppose $\varphi$ extends to a homomorphism $\psi: E \to J$. Define $f$ by $f((j,e)K) = (j\psi(e), eM)$. This map is well defined: if $(j,e)K = (j',e')K$, then $(j,e)(a\varphi(m)a^{-1}, b m^{-1} b^{-1}) = (j',e')$ for some $a \in J$, $b \in E$. Since $J$ is abelian, $j\varphi(m) = j'$, and because $M$ is central, $e b m^{-1} b^{-1} = e'$ implies $e m^{-1} = e'$, so $eM = e'M$ and $m = e'^{-1}e$. Consequently, $j\psi(e) = j'\psi(e')$, giving $(j\psi(e), eM) = (j'\psi(e'), e'M)$.
\end{proof}
A crucial observation is that for a square-central extension, i.e., when $M \leq E$ satisfies $M^2 = 1$ and $M \subseteq Z(E)$, there is a canonical $\FF_2$-module structure on $M$ given by $0\cdot x = 0$ and $1\cdot x = x$ for $x \in M$. Moreover, morphisms of $\FF_2$-modules coincide with group homomorphisms between abelian groups. Here, we also denote by $E^2$ the subgroup of squares.

\begin{lem} Consider a square-central extension
$$1 \rightarrow M \rightarrow E \rightarrow E/M \rightarrow 1\,,$$
with $|E^2\cap M|<\infty$. 
 Then there is an $\FF_2$-module isomorphism 
    \[E^2\cap M \cong \op{im}(\op{Hom}_{\FF_2}(M,\mathbb{F}_2)\rightarrow \xi^s(E/M,\mathbb{F}_2))\,.\]
\end{lem}

\begin{proof}
Lemma \ref{lemma 15} yields an $\FF_2$-module isomorphism
$$\op{im}(\op{Hom}_{\mathbb{F}_2}(M,\mathbb{F}_2)\rightarrow \xi^s(E/M,\mathbb{F}_2)) \cong \frac{\op{Hom}_{\mathbb{F}_2}(M,\mathbb{F}_2)}{\op{im}(\op{Hom}(E,\mathbb{F}_2)\rightarrow \op{Hom}_{\mathbb{F}_2}(M,\mathbb{F}_2))}\,.$$ 
We claim that $L:=\op{im}(\op{Hom}(E,\mathbb{F}_2)\rightarrow \op{Hom}_{\mathbb{F}_2}(M,\mathbb{F}_2))$ coincides with
$$K_0:=\op{ker} (\op{Hom}_{\FF_2}(M,\mathbb{F}_2)\rightarrow \op{Hom}_{\FF_2}(E^2 \cap M,\mathbb{F}_2) )\,.$$ This implies $\op{Hom}_{\mathbb{F}_2}(M,\mathbb{F}_2)/L\cong \op{Hom}_{\mathbb{F}_2}(E^2 \cap M,\mathbb{F}_2)$, and the lemma follows. 

We prove the two inclusions. First, take $\varphi \in L$; then $\varphi$ extends to $E$. Since $\mathbb{F}_2$ is an abelian group with trivial square, we have $E^2 \cap M \subset \ker \varphi$. Hence $\varphi \in K_0$, and $L \subset K_0$.

Conversely, let $\varphi \in K_0$, so $E^2 \cap M \subset \ker \varphi$. Then $\varphi$ factors through $M/(E^2 \cap M)$. This quotient is isomorphic to $E^2M/E^2$, which is an $\FF_2$-submodule of $E/E^2$. Because $\mathbb{F}_2$ is divisible as an $\FF_2$-module, the morphism extends to $E/E^2$ as in the following diagram:
$$\xymatrix{E^2M/E^2=M/(E^2\cap M)\ar[r]\ar@{_(->}[d] &\mathbb{Z}/2\\E/E^2 \ar[ru]&} \,,$$
Composing with the canonical projection $E \to E/E^2$ gives an extension of $\varphi$ to $E$. Thus $K_0 \subset L$.
 \end{proof}

For a presentation $G = \langle F \mid R \rangle$, we have the square-central extension \eqref{gigon}. We now show that the transgression map is surjective.
 
\begin{lem}
   The transgression map $\op{Hom}(R/[F,R]R^2,A)\rightarrow \xi^s(G,A)$ is surjective.
\end{lem}

\begin{proof}
Let $\epsilon \in \xi^s(G,A)$ be represented by the square-central extension
$$
1 \rightarrow  A \rightarrow G^* \stackrel{\eta}{\rightarrow}  G \rightarrow   1\,.
$$
Since $F$ is free, there exists a homomorphism $f: F \to G^*$ making the following diagram commute:
$$\xymatrix{ F \ar[r]\ar[d]_f& G \ar[d]^{\op{id}} \\
G^* \ar[r]^\eta          & G      } $$    
Note that $R \subset \ker(f)$. Moreover, $[F,R]R^2 \subset \ker(f)$ because for $x \in F$ and $y \in R$, we have $f([x,y]) = [f(x), f(y)]$; here $f(y) \in \ker(\eta) = \operatorname{im}(A)$, and the square-central property of $A$ gives $[f(x), f(y)] = 1$ and $f(y)^2 = 1$. Thus $f$ induces a homomorphism $f: F/[F,R]R^2 \to G^*$, yielding a commutative diagram $$\xymatrix{1 \ar[r] & R/[F,R]R^2 \ar[r] \ar[d]_{\tilde{f}}& F/[F,R]R^2 \ar[r] \ar[d]_f& G \ar[r]\ar[d]_{\op{id}} & 1\\ 1 \ar[r] & A \ar[r] & G^* \ar[r] & G \ar[r] & 1,}$$ where $\tilde{f}: R/[F,R]R^2 \to A$ is the restriction. The transgression of $\tilde{f}$ coincides with the original class $\epsilon \in \xi^s(G,A)$, proving the lemma. 
\end{proof}
Applying the universal coefficient theorem with coefficients in $\ZZ_2$ yields
$$H^2(G;\ZZ_2)\cong H_1(G;\ZZ)\otimes \ZZ_2\oplus H_2(G;\ZZ)\otimes \ZZ_2\,.$$
As an application, we compute these groups for cyclic and dihedral groups: 
$$H^2(\ZZ_n;\ZZ_2)=\left\{\begin{array}{cc}
    0     & n \textrm{ odd}, \\\
   \ZZ_2 & n \textrm{ even}, \\
 \end{array}\right.
 \hspace{1cm}
 H^2(D_{2n};\ZZ_2)=\left\{\begin{array}{cc}
    \ZZ_2     & n \textrm{ odd}, \\\
   \ZZ_2\oplus\ZZ_2\oplus\ZZ_2 & n \textrm{ even}. \\
 \end{array}\right.
 $$
For the cyclic group $\ZZ_n = \langle x : x^n = 1 \rangle$ with $n$ even, the generator corresponds to $(x^{n/2})$; see Proposition \ref{cyclic}. For the dihedral groups $D_{2n} = \langle a,b : a^2 = b^2 = (ab)^n = 1\rangle$ with $c := ab$, when $n$ is odd the generator is $(a)$, while for $n$ even the generators are $(c^{n/2})$, $(a)$, and $(ac)$.

The symmetric groups $S_n$ satisfy
$$H^2(S_n;\ZZ_2)=\left\{
\begin{array}{cr}
    0     & n=1, \\
   \ZZ_2 & n=2,3, \\
   \ZZ_2\oplus\ZZ_2& n>3\,.\\
 \end{array}
\right.$$
If $S_n$ is generated by transpositions $(ij)$ with $i,j \in {1,\dots,n}$, then a generator is $((12))$ for $n = 2,3$, and the two generators are $((12))$ and $((12),(34))$ for $n > 3$.

\section{The unoriented Bogomolov multiplier}
\label{bogo}

The Bogomolov multiplier originally arose as an obstruction to the rationality\footnote{An algebraic variety is considered rational if it possesses a Zariski open subset isomorphic to an open subset of some projective space.} of an algebraic variety $X$ over $\mathbb{C}$.
Bogomolov \cite{Bo87} studied the quotient variety $X = V/G$, where $V$ is a faithful linear representation of a linear algebraic group $G$ over $\CC$ (we restrict to the case where $G$ is a finite group). Saltman \cite{Sal84} and Bogomolov \cite{Bo87} provided the following explicit description:
\begin{equation}
    B_0(G)=\bigcap_{A}\op{ker}\{\op{res}_G^A:H^2(G;\CC^*)\rightarrow H^2(A;\CC^*)\}\,,
\end{equation}
where $A$ ranges over all bicyclic subgroups of $G$, i.e., subgroups generated by two commuting elements.
The rationality of $V/G$ is equivalent to the condition that the field of invariants $\CC(G)^G$ is a pure transcendental extension of the constant field $\CC$; this is the classical Noether problem \cite{Sw83}.
Moravec \cite{Mo12} gave another description of $B_0(G)$, based on the bordism interpretation of the Schur multiplier due to Miller \cite{Mi52}. This description considers the quotient of $\mathcal{M}(G)$ by the subgroup generated by extendable $G$-cobordisms over the torus, or equivalently, the classes of principal $G$-bundles with base space a disjoint union of tori (free $G$-actions on a disjoint union of tori). These are extendable because the free action of $G$ extends to a (not necessarily free) action on a three-manifold. In summary, we obtain the isomorphism
\begin{equation}
    B_0(G)\cong \frac{\mathcal{M}(G)}{\langle \textrm{toral classes}\rangle}\,.
\end{equation}
Thus it is natural to define an analogous invariant in the unoriented setting as follows:
\begin{defn}
For a finite group $G$, the \textbf{unoriented Bogomolov multiplier}, denoted $B_0(G;\ZZ_2)$, is the quotient of the unoriented Schur multiplier $\mathcal{M}(G;\ZZ_2)$ by the subgroup generated by the $G$-cobordisms over the torus $\langle x,y\rangle$, the Klein bottle $(x',y')$, and the projective space $(z)$, subject to $[x,y]=1$, $\{x',y'\}=1$, and $z^2=1$.
\end{defn}

In Section \ref{hopf}, we established the Hopf formula for the unoriented Schur multiplier, yielding $\mathcal{M}(G;\ZZ_2) \cong H^2(G;\ZZ_2)$. The universal coefficient theorem then gives
$$\mathcal{M}(G;\ZZ_2)\cong (\mathcal{M}(G)\otimes \ZZ_2)\oplus (G/G'\otimes \ZZ_2)\,.$$
where the second term $G/G'\otimes \ZZ_2\cong 2\text{-}\operatorname{tor}(G/G')$ consists of the $2$-torsion elements of the abelianization of $G$.
These elements are represented by $G$-cobordisms over the Klein bottle of the form $(1,g)$ for $g\in G$. By relation \eqref{four20} in Theorem \ref{dador}, these elements reduce to the projective elements $(1,g)\sim (g)$. As shown in Proposition \ref{prC1}, such elements are extendable. Consequently, the unoriented Bogomolov multiplier reduces to the quotient of $\mathcal{M}(G)\otimes \ZZ_2$ by the subgroup generated by the toral classes. Thus we obtain the isomorphism
\begin{equation}\label{inme}
    B_0(G;\ZZ_2)\cong \frac{\mathcal{M}(G)\otimes \ZZ_2}{\langle\textrm{toral classes}\rangle}\,.
\end{equation}
An immediate consequence of this identification since $\mathcal{M}(G)\otimes \ZZ_2\cong \mathcal{M}(G)/2\mathcal{M}(G)$ is the following proposition:
\begin{prop}\label{guero} If a class $[S,G]\in B_0(G)$ is trivial, then the image on $B_0(G;\ZZ_2)$ is also trivial. 
\end{prop}

The previous discussion implies that when a group $G$ has a Schur multiplier which is an elementary abelian $2$-group, then $B_0(G) = B_0(G;\mathbb{Z}_2)$. This property holds for extraspecial $2$-groups of order $2^{2n+1}$ with $n > 1$, where the Schur multiplier is an elementary abelian $2$-group of order $2^{2n^2 - n - 1}$; see \cite[p. 138]{Kar87}. 

The following theorem, which is the unoriented counterpart of the main result in \cite{ASSU23}, will be proved in \cite{Seg}:

\begin{thm}\label{no-proof}
A free action of a finite group $G$ on an unoriented closed surface $S$ extends to a (not necessarily free) action on a $3$-manifold if and only if the class $[S,G]\in B_0(G;\ZZ_2)$ in the unoriented Bogomolov multiplier is trivial.
\end{thm}

This raises an open question about rationality in the unoriented setting, or equivalently, about the appropriate formulation of the Noether problem for which the unoriented Bogomolov multiplier $B_0(G;\ZZ_2)$ serves as an obstruction.

As a consequence of Section \ref{schur}, we obtain that $B_0(G;\ZZ_2)$ is trivial for any abelian group; see Theorem \ref{abeliano}. Similarly, Propositions \ref{cyclic} and \ref{dihedral} show that the unoriented Bogomolov multiplier is trivial for all cyclic and dihedral groups. At the end of Section \ref{hopf}, we observed that the generators of the Schur multiplier of the symmetric group are $G$-cobordisms over the projective space and the Klein bottle. Consequently, $B_0(G)$ is trivial for symmetric groups. These results constitute the unoriented counterpart of those presented in \cite{DS22}. 
Moreover, as the reader may observe, these facts follow directly from Proposition \ref{guero} and \cite{DS22}, and they can be extended to include the case of alternating groups.

\begin{thm} For finite abelian, dihedral, symmetric groups, and alternating groups the unoriented Bogomolov multiplier is trivial. 
\end{thm}

We conclude this section by analyzing an example where the unoriented Bogomolov multiplier is nontrivial.
In \cite{ASSU23}, two groups of orders $64=2^6$ and $243=3^5$ are shown to have nontrivial Bogomolov multiplier. The group of order $243$ is Samperton's counterexample to the free extension conjecture; see \cite{Sam22}. Since any group of odd order has trivial $H^2(G;\ZZ_2)$, Samperton's counterexample becomes extendable when considering actions on unoriented $3$-manifolds.

\begin{example}\label{64}
Consider the semidirect product $\ZZ_8\rtimes Q_8$, where $Q_8$ denotes the quaternion group. Recall the presentations $Q_8=\langle a,b \mid a^2=b^2,; aba^{-1}=b^{-1}\rangle$ and $\ZZ_8=\langle c \mid c^8=1\rangle$. The action of $Q_8$ on $\ZZ_8$ is given by
$$a\cdot c=c^3,\hspace{0.3cm}b\cdot c=c^5,\hspace{0.3cm}(ab)\cdot c=c^7\,.$$
Thus $\ZZ_8\rtimes Q_8$ admits the presentation
$$\left\langle a,b,c: a^2=b^2,aba^{-1}=b^{-1},c^8=1,aca^{-1}=c^3,bcb^{-1}=c^5\right\rangle\,.$$
This group is identified as $\mathsf{SmallGroup(64,182)}$ in \cite{gap}.
The generator of the Bogomolov multiplier $B_0(\ZZ_8\rtimes Q_8)$ is \begin{equation}\label{nogen}\langle a,c\rangle \langle ab,c\rangle\,,\end{equation}
where $[a,c][ab,c]=aca^{-1}c^{-1}(ab)c(ab)^{-1}c^{-1}=c^3c^{-1}c^7c^{-1}=1$. 
The integral homology groups are
$$H_1(\ZZ_8\rtimes Q_8;\ZZ)=\ZZ_2\oplus \ZZ_2\oplus\ZZ_2\,,\hspace{1cm}H_2(\ZZ_8\rtimes Q_8;\ZZ)=\ZZ_2\,.$$
Using \cite{gap}, one can verify that the first integral homology are the image of two subgroups isomorphic to $\ZZ_4$ and one $\ZZ_8$. Moreover, the image of the second integral homology of every proper subgroup is trivial. 
Consequently, the class \eqref{nogen} remains nontrivial in the unoriented Bogomolov multiplier $B_0(\ZZ_8\rtimes Q_8;\ZZ_2)$. The reader may observe that, by applying the isomorphism \eqref{inme} and that $B_0(G)$ is an elementary abelian 2-group, one immediately concludes that the unoriented Bogomolov multiplier is nontrivial.
\end{example}

During the preparation of this work the authors used the AI deepseek in order to revise the English grammar and style of the article. After using this AI deepseek, the authors reviewed and edited the content as needed and take full responsibility for the content of the published article.

\end{document}